\documentclass[12pt, reqno]{amsart}

\usepackage[utf8]{inputenc}
\usepackage[T1]{fontenc}
\usepackage{amsthm}
\usepackage{amsmath}
\usepackage{amssymb}
\usepackage[inline]{enumitem}
\usepackage[dvips]{graphicx}
\usepackage{color}
\usepackage{hyperref}
\usepackage{url}
\usepackage{bm}
\usepackage{stackrel}
\usepackage[left=3cm, right=3cm, top=2.5cm, bottom=3cm]{geometry}
\usepackage{fancyhdr}
\usepackage{mathrsfs}
\usepackage{stmaryrd}
\usepackage{soul}
\usepackage{nicefrac}
\usepackage{datetime}
\usepackage{comment}
\usepackage[normalem]{ulem}

\usepackage{tikz}
\usetikzlibrary{decorations.pathreplacing,matrix,arrows,positioning,automata,shapes,shadows,calc,fadings,decorations,snakes,through,intersections}

\longdate

\AtBeginDocument{%
   \def\MR#1{}
}

\newtheorem{theorem}{Theorem}[section]

\theoremstyle{definition}

\newtheorem{question}[theorem]{\textbf{Question}}

\newtheorem{claim}{\textsc{Claim}}

\hypersetup{
    pdftitle={On the number of gaps of sequences with Poissonian Pair Correlations},
    pdfauthor={Paolo Leonetti},
    pdfmenubar=false,
    pdffitwindow=true,
    pdfstartview=FitH,
    colorlinks=true,
    linkcolor=blue,
    citecolor=green,
    urlcolor=cyan
}

\renewcommand{\emptyset}{\varnothing}

\renewcommand{\rho}{\varrho}

\providecommand{\NNb}{\mathbf{N}}

\providecommand{\RRb}{\mathbf{R}}

\renewcommand{\xi}{a}

\begin{document}

\title{On the number of gaps of sequences with Poissonian Pair Correlations}

\author{Christoph Aistleitner}
\address{Institute of Analysis and Number Theory, Graz University of Technology | Kopernikusgasse 24/II, 8010 Graz, Austria}
\email{aistleitner@math.tugraz.at}

%
\author{Thomas Lachmann}
\address{Institute of Analysis and Number Theory, Graz University of Technology | Kopernikusgasse 24/II, 8010 Graz, Austria}
\email{lachmann@math.tugraz.at}
%

\author{Paolo Leonetti}
\address{Institute of Analysis and Number Theory, Graz University of Technology | Kopernikusgasse 24/II, 8010 Graz, Austria}
\email{leonetti.paolo@gmail.com}

%
\author{Paolo Minelli}
\address{Institute of Analysis and Number Theory, Graz University of Technology | Kopernikusgasse 24/II, 8010 Graz, Austria}
\email{minelli@math.tugraz.at}

\thanks{C.A. is supported by the Austrian Science Fund (FWF), projects F-5512, I-3466 and Y-901. T.L. is supported by FWF project Y-901. P.L. is supported by FWF project F-5512. P.M. is supported by FWF project I-3466.}




\begin{abstract}
\noindent{} A sequence $(x_n)$ on the torus is said to have Poissonian pair correlations if 
$
\# \{1\le i\neq j\le N: |x_i-x_j| \le s/N\}=2sN(1+o(1))
$ 
for all reals $s>0$, as $N\to \infty$. 

It is known that, if $(x_n)$ has Poissonian pair correlations, then the number $g(n)$ of different gap lengths between neighboring elements of $\{x_1,\ldots,x_n\}$ cannot be bounded along every index subsequence $(n_t)$. 
First, we improve this by showing that the maximum among the multiplicities of the neighboring gap lengths of $\{x_1,\ldots,x_n\}$ is $o(n)$, as $n\to \infty$. 
%
Furthermore, we show that, for every function $f: \mathbf{N}^+\to \mathbf{N}^+$ with $\lim_n f(n)=\infty$, there exists a sequence $(x_n)$ with Poissonian pair correlations and such that 
$g(n) \le f(n)$ for all sufficiently large $n$. 
This answers negatively a question posed by G. Larcher.
\end{abstract}
\maketitle
\thispagestyle{empty}

\section{Introduction}\label{sec:introduction}

Let $x=(x_n)$ be a sequence on the torus, hereafter identified with the interval $[0,1)$. For every positive integer $N$ and real $s>0$, define
$$
F_{x,N}(s):=\frac{1}{N}\#\left\{1\le i\neq j\le N: \,|x_i-x_j| < \frac{s}{N}\right\},
$$
where $|\cdot|$ stands for the distance from the nearest integer, that is, $|z|=\min(z,1-z)$ for all $z \in [0,1)$. 
The sequence $x$ is said to have \emph{Poissonian pair correlations} if 
$$
\lim_{N\to \infty} F_{x,N}(s)=2s
$$
for all $s>0$. The original motivation for the study of sequences with Poissonian pair correlations comes from quantum physics, see \cite{MR3822226, MR3736514, LarcherMathProc} and references therein. It has been recently shown that this is a stronger notion than the classical uniform distribution, the converse being false in general, see \cite{MR3703937, MR3673633, Marklof2019, MR3704235}. We recall that there are only a couple of "explicit" sequences for which it could be proved that they have Poissonian pair correlations, see \cite{MR3340360, MR3336607}. 


Given a sequence $(x_n)$ and an integer $N\ge 2$, let $G(N)$ be the set of different gap lengths between neighboring elements of $\{x_1,\ldots,x_N\}$, that is, 
$$
G(N):=\{r \in [0,1): r=|x_{\sigma(k+1)}-x_{\sigma(k)}| \,\,\text{ for some }k=1,\ldots,N\},
$$
where $\sigma: \{1,\ldots,N\} \to \{1,\ldots,N\}$ is a permutation such that $x_{\sigma(1)}\le \cdots \le x_{\sigma(N)}$ and ${\sigma(N+1)}:={\sigma(1)}$. Set $g(N):=\# G(N)$ and 
let $\{\ell_{N,1},\ldots,\ell_{N,g(N)}\}$ be the elements of $G(N)$ in increasing order, so that $\ell_{N,1}<\cdots<\ell_{N,g(N)}$. For each $i=1,\ldots,g(N)$, let $\varphi_{N,i}$ be the number of gaps of length $\ell_{N,i}$, that is, 
$$
\varphi_{N,i}:=\#\{j \in \{1,\ldots,N\}: |x_{\sigma(j+1)}-x_{\sigma(j)}|=\ell_{N,i}\}.
$$


The following result has been show in \cite{LarcherNegative}, cf. also \cite[Theorem 1]{LarcherSurvey}:
\begin{theorem}\label{thm:original_larcher}
Let $(x_n)$ be a sequence with Poissonian pair correlations. Then 
\begin{equation}\label{eq:finitegappropertylarcher}
\liminf_{n\to \infty} g(n)=\infty,
\end{equation}
that is, there is no subsequence $({n_t})$ of indexes with a finite number of distinct gap lengths between neighboring elements.
\end{theorem}


%

Note that, for each $N\ge 2$, we have $\sum_{i\le g(N)} \ell_{N,i} \varphi_{N,i}=1$ and $\sum_{i\le g(N)}\varphi_{N,i}=N$. 
In particular, we have $\max_{i\le g(N)}\varphi_{N,i} \ge \frac{N}{g(N)}$, that is, 
\begin{equation}\label{eq:inequalitybasic}
g(N) \ge \frac{N}{\max_{i\le g(N)}\varphi_{N,i}}.
\end{equation}

The aim of this article is twofold: first, we prove a more general version of Theorem \ref{thm:original_larcher}, by showing that also the right-hand side of \eqref{eq:inequalitybasic} is divergent. 
\begin{theorem}\label{thm:varphinotlarge}
Let $(x_n)$ be a sequence with Poissonian pair correlations. Then 
\begin{equation}\label{eq:claimsmallo}
\max_{i\le g(n)} \varphi_{n,i}=o(n)
\end{equation}
as $n\to \infty$, that is, there is no subsequence $({n_t})$ of indexes and constant $c>0$ for which at least one number of distinct gap lengths is $\ge cn_t$.
\end{theorem}
As an immediate consequence of Theorem \ref{thm:varphinotlarge} and the Three Gap Theorem \cite{MR3706822}, we obtain that for every $\alpha \in \RRb$, the Kronecker sequence $(\alpha n)$ does not have Poissonian pair correlations, cf. also \cite{LarcherNegative}.

Secondly, during the open problems session of the Workshop and Winter School on Local Statistics of Point Sequences (Linz, 2019), Gerhard Larcher asked whether Theorem \ref{thm:original_larcher} could be extended as it follows, cf. also \cite[Problem 4]{LarcherSurvey}:
\begin{question}\label{q:larcherlinz}
Does there exist a "slowly-growing" function $f: \mathbf{N}^+\to \mathbf{N}^+$ with $\lim_{n\to \infty}f(n)=\infty$ such that, if $(x_n)$ has Poissonian pair correlation, then necessarily $g(n)\ge f(n)$ for all $n\ge 2$? For instance, is it true that if $g(n)\le \log\log n$ for infinitely many $n$, then $(x_n)$ does not have Poissonian pair correlations?
\end{question}
It is known that almost all sequences have Poissonian pair correlations, see e.g. \cite{HinrichsEco} and \cite{SteinerbergerRd}. In addition, it is easy to see that almost all sequences in $[0,1)$ have all different gap lengths between neighboring elements. 
This implies that, with probability $1$, a sequence $(x_n)$ has Poissonian pair correlations and $g(n)=n$ for all $n \in \NNb^+$.

We show, in a strong sense, that the answer to Question \ref{q:larcherlinz} is negative.
\begin{theorem}\label{thm:answerlarcher}
Fix a function $f: \mathbf{N}^+\to \mathbf{N}^+$ with $\lim_{n\to \infty}f(n)=\infty$. Then there exists a sequence $(x_n)$ with Poissonian pair correlations such that $g(n)\le f(n)$ for all sufficiently large $n$.
\end{theorem}
The proof of Theorem \ref{thm:answerlarcher} follows in Section \ref{sec:profmain}.


\subsection{Notations.} We employ the Landau--Bachmann ``Big Oh'' notation $O$ and the associated Vinogradov symbols $\ll$ and $\gg$, and the "small oh" notation $o$. In addition, $\mathbf{N}^+$ and $\RRb^+$ stand for the sets of positive integers and positive reals, respectively. Lastly, given $A\subseteq \RRb$ and $x \in \RRb$, let $\bm{1}_A$ be the indicator function of $A$, that is, $\bm{1}_A(x)=1$ if $x \in A$, and $0$ otherwise.

%


\section{Proof of Theorem \ref{thm:varphinotlarge}}

Let us assume for the sake of contradiction that \eqref{eq:claimsmallo} does not hold, i.e., 
$$
\delta:=\limsup_{n\to \infty}\,\frac{\max_{i\le g(n)}\varphi_{n,i}}{n}>0.
$$
Fix a constant $c \in (0,\delta)$. Then there exist a strictly increasing sequence of positive integers $({n_t})$ and an integer sequence $(i_t)$ such that 
$$
i_t \in \{1,\ldots,g(n_t)\}\,\,\,\text{ and }\,\,\,\varphi_{n_t,i_t} \ge c n_t
$$
for all $t \in \NNb^+$. 
Hence, define 
$
\alpha:=\limsup_{t\to \infty} n_t\, \ell_{n_t,i_t}
$ 
and note that $\alpha$ is finite. Indeed, in the opposite, there would exist $t \in \NNb^+$ such that $n_t\,\ell_{n_t,i_t}\ge 2/c$, from which we obtain the contradiction
$$
2\le \ell_{n_t,i_t} \varphi_{n_t,i_t} \le \sum_{i\le g(n_t)}\ell_{n_t,i}\varphi_{n_t,i}=1.
$$
Fix $\beta>\alpha$. It follows that there exists $t_0 \in \NNb^+$ such that $\ell_{n_t,i_t}<\frac{\beta}{n_t}$ for all $t\ge t_0$. 

At this point, fix $m \in \NNb^+$ and define the set
$$
Q_{j,t}:=\left[\frac{(j-1)\beta}{mn_t},\,\frac{j\beta}{mn_t}\right)
$$
for $j=1,\ldots,m$ and $t \in \NNb^+$. Hence $\{Q_{j,t}: j=1,\ldots,m\}$ is a partition of $[0,\beta/n_t)$ for all $t \in \NNb^+$. Therefore there exist $j_m \in \{1,\ldots,m\}$ and an infinite set $T \subseteq \NNb^+$ such that $\ell_{n_t,i_t} \in Q_{j_m,t}$ for all $t \in T$.

It follows by construction that
\begin{displaymath}
\begin{split}
n_tF_{x,n_t}\left(\frac{j_m\beta}{m}\right)&-n_tF_{x,n_t}\left(\frac{(j_m-1)\beta}{m}\right)\\
&\ge \#\left\{1\le i\neq j\le n_t: |x_i-x_j| \in Q_{j_m,t}\right\}\\
&\ge 2\varphi_{n_t,i_t} \ge 2cn_t
\end{split}
\end{displaymath}
for all $t \in T$. Considering that the sequence $(x_n)$ has Poissonian pair correlations, we conclude, dividing by $2n_t$ and letting $t\to \infty$ (with $t\in T$), that 
$
\beta/m\ge c.
$ 
However, this is impossible whenever $m$ is sufficiently large.


\section{Proof of Theorem \ref{thm:answerlarcher}}\label{sec:profmain}

The main idea in the proof of Theorem \ref{thm:answerlarcher} is to construct a sequence of jointly independent random variables, split in deterministic blocks and random blocks, such that each one takes values in rational numbers having suitable powers of $2$ as denominators. 
Then, the cardinality of the random part will be sufficiently large to deduce that the overall sequence has Poissonian pair correlations. 
At the same time, the deterministic part will be sufficiently small not to affect the Poissonian pair correlations property, but sufficiently large to control the number of distinct gaps of the sequence. 

\begin{proof}[Proof of Theorem \ref{thm:answerlarcher}]
For all $m \in \NNb^+$, define $I_m:=(2^{m-1},2^m] \cap \NNb^+$, and set, by convention, $I_0:=\{1\}$. Let $a: \NNb^+\to \NNb^+$ be a weakly increasing function (that is, $a(n) \le a(n+1)$ for all $n\in \NNb^+$) with 
$\lim_{n\to \infty} a(n)=\infty$ 
that will be chosen later. Moreover, let $X=(X_1,X_2,\ldots)$ be a sequence of jointly independent random variables on a probability measure space $(\Omega, \mathscr{F}, \mathrm{P})$ such that, for each $m \in \NNb^+$ and for each $i \in I_m$, $X_i$ has uniform distribution on 
\begin{equation}\label{am}
A_m:=
\left\{\frac{j}{2^{m+a(m)}}: j=0,1,\ldots,2^{m+a(m)}-1\right\},
\end{equation}
and $X_1(\omega):=0$ for all $\omega \in \Omega$ (hence, the random points $X_i$ are sampled on a grid of points with denominators which are a power of $2$, where the size of the denominator increases relatively to $2^i$ as $i$ increases). 
We fix also a  positive real sequence $y=(y_n)$ such that 
$
y_n=n+o(n)
$ 
as $n\to \infty$, and we define 
$$
\tilde{F}_{x,y,N}(s):=\frac{1}{N}\#\left\{1\le i\neq j\le N: \,|x_i-x_j| \le \frac{s}{y_N}\right\}
$$
for all real $s>0$ and $N \in \NNb^+$. In particular, $F_{x,N}=\tilde{F}_{x,y,N}$ provided that $y$ is the identity sequence.


\begin{claim}\label{claim:expectedvalue}
\emph{\textup{(}Expected values of random components.\textup{)}}
Fix $s \in \RRb^+$. 
Then 
$$
\mathbb{E}[\tilde{F}_{X,y,N}(s)]:=\int_\Omega \tilde{F}_{X(\omega),y,N}(s)\, \mathrm{P}(\mathrm{d}\omega)=2s+o(1),\,\,\,\,\,\,\text{ as }N\to \infty.
$$
\end{claim}
\begin{proof}
To start, we have
\begin{displaymath}
\begin{split}
\mathbb{E}[\tilde{F}_{X,y,N}(s)]&=\frac{2}{N}\int_\Omega \#\left\{1\le i<j\le N: \,|X_i(\omega)-X_j(\omega)| \le \frac{s}{y_N}\right\}\,\mathrm{P}(\mathrm{d}\omega)\\
&=\frac{2}{N}\int_\Omega \, \left(\sum_{1\le i<j\le N} \bm{1}_{\left[-\frac{s}{y_N},\frac{s}{y_N}\right]}(X_i(\omega)-X_j(\omega))\right)\,\mathrm{P}(\mathrm{d}\omega)\\
&=\frac{2}{N}\sum_{1\le i<j\le N} \mathbb{E}\left[\bm{1}_{\left[-\frac{s}{y_N},\frac{s}{y_N}\right]}(X_i-X_j)\right].
\end{split}
\end{displaymath}
Note that, if $i<j$ then by the independence assumption $X_i-X_j$ has the same distribution as $X_j$, so that
\begin{displaymath}
\begin{split}
\mathbb{E}[\tilde{F}_{X,y,N}(s)]&=\frac{2}{N}\sum_{1\le i<j\le N}\mathrm{P}\left(X_j \in \left[-\frac{s}{y_N},\frac{s}{y_N}\right]\right)\\
&=\frac{2}{N}\sum_{1\le j\le N}(j-1)\,\mathrm{P}\left(X_j \in \left[-\frac{s}{y_N},\frac{s}{y_N}\right]\right).
\end{split}
\end{displaymath}
If $j \in I_k$, for some $k$ sufficiently large, let us say $k\ge k_0$, we have
\begin{equation}\label{eq:lambdak}
\gamma_k:=\mathrm{P}\left(X_j \in \left[-\frac{s}{y_N},\frac{s}{y_N}\right]\right)=\frac{2\lfloor 2^{k+a(k)}s/y_N \rfloor+1}{2^{k+a(k)}}.
\end{equation}
Hence, setting $m:=\lfloor \log_2(N)\rfloor$, we obtain
\begin{displaymath}
\begin{split}
\mathbb{E}[\tilde{F}_{X,y,N}(s)]&=\frac{2}{N}\left(O(1)+\sum_{1\le k\le m}\sum_{j \in I_k}(j-1)\,\gamma_k+\sum_{2^m+1\le j\le N}(j-1)\,\gamma_{m+1}\right),
\end{split}
\end{displaymath}
where the last sum is $0$ if $N=2^m$. Considering that $\sum_{j \in I_k}(j-1)=2^{k-2}(3\cdot 2^{k-1}-1)$ for all $k \in \NNb^+$, we get
\begin{equation}\label{eq:s1s2_estimate}
\begin{split}
\mathbb{E}[\tilde{F}_{X,y,N}(s)]
&=o(1)+\frac{2}{N}\left(\sum_{1\le k\le m}2^{k-2}\gamma_k(3\cdot 2^{k-1}-1)+\gamma_{m+1}\sum_{2^m\le j\le N-1}j\,\right)\\
&=o(1)+\frac{2}{N}(S_1+S_2),
\end{split}
\end{equation}
where $S_1:=\sum_{1\le k\le m}2^{k-2}\gamma_k(3\cdot 2^{k-1}-1)$ and $S_2:=\gamma_{m+1}\sum_{2^m\le j\le N-1}j$. 

At this point, the first sum can be rewritten as 
\begin{displaymath}
\begin{split}
S_1&=\sum_{1\le k\le m}(3\cdot 2^{k-1}-1)\cdot \frac{2\lfloor 2^{k+a(k)}s/y_N \rfloor+1}{2^{a(k)+2}}\\
&=3\sum_{1\le k\le m}2^k\cdot \frac{2\lfloor 2^{k+a(k)}s/y_N \rfloor+1}{2^{a(k)+3}}-\sum_{1\le k\le m}\frac{2\lfloor 2^{k+a(k)}s/y_N \rfloor+1}{2^{a(k)+2}}\\
&=\frac{3s}{y_N}\sum_{1\le k\le m}2^{2k-2}+O\left(\sum_{1\le k\le m}\frac{2^{k}}{2^{a(k)}}\right)+O\left(\sum_{1\le k\le m}\frac{2^{k}}{N}\right)\\
&=\frac{s}{y_N}(4^m-1)+O\left(\sum_{1\le k\le m}\frac{2^{k}}{2^{a(k)}}\right)+O\left(1\right)\\
&=s\cdot \frac{4^m}{y_N}+o\left(N\right)+O(1),
\end{split}
\end{displaymath}
where the last $o(N)$ follows by the fact that $\sum_{1\le k\le m}\frac{2^{k}}{2^{m}}\cdot \frac{1}{2^{a(k)}}\to 0$ as $m\to \infty$ (indeed if $(z_n)$ is a real sequence which is convergent to $0$ then $\left(\sum_{1\le i\le n}\frac{z_{n+1-i}}{2^i}\right)$ is convergent to $0$ as well). 
Hence
\begin{equation}\label{eq:s1_estimate}
S_1=s\cdot \frac{4^m}{y_N}+o(N).
\end{equation}

Similarly, we have
\begin{displaymath}
\begin{split}
S_2&
=\frac{2\lfloor 2^{m+1+a(m+1)}s/y_N \rfloor+1}{2^{m+1+a(m+1)}}\sum_{2^m\le j\le N-1}j\\
&=\left(\frac{2s}{y_N}+o\left(\frac{1}{N}\right)\right)\left(\binom{N}{2}-\binom{2^m}{2}\right)\\
&=\frac{2s}{y_N}\left(\binom{N}{2}-\binom{2^m}{2}\right)+o(N),
\end{split}
\end{displaymath}
which implies that
\begin{equation}\label{eq:s2_estimate}
S_2=\frac{s}{y_N}(N^2-4^m)+o(N).
\end{equation}

Putting together \eqref{eq:s1s2_estimate}, \eqref{eq:s1_estimate}, and \eqref{eq:s2_estimate}, and recalling that $y_n=n+o(n)$ by hypothesis, 
we obtain that 
\begin{displaymath}
\begin{split}
\mathbb{E}[\tilde{F}_{X,y,N}(s)]&=\frac{2s}{N}\left(\frac{4^m}{y_N}+\frac{N^2-4^m}{y_N}+o(N)\right)+o\left(1\right)\\
&=2s\cdot \frac{N}{y_N}+o(1)=2s+o(1),
\end{split}
\end{displaymath}
which concludes the proof.
\end{proof}


\begin{claim}\label{claim:variance} 
\emph{\textup{(}Bounding the variances of random components.\textup{)}} 
Fix $s \in \RRb^+$. Then 
$$\mathrm{Var}[\tilde{F}_{X,y,N}(s)]\ll 1/N,\,\,\,\,\,\,\text{ as }N\to \infty.$$
\end{claim}
\begin{proof}
Note that, since $(X_n)$ is a sequence of jointly independent random variables, then also the measurable transformations $\bm{1}_{\left[-\frac{s}{y_N},\frac{s}{y_N}\right]}(X_{i_1}-X_{j_1})$ and $\bm{1}_{\left[-\frac{s}{y_N},\frac{s}{y_N}\right]}(X_{i_2}-X_{j_2})$ are independent for all $i_1<j_1$ and $i_2<j_2$ such that $(i_1,j_1) \neq (i_2,j_2)$, cf. e.g. \cite[Corollary 272L]{MR2462280}.  Considering, as in the proof of Claim \ref{claim:expectedvalue}, that $X_i-X_j$ has the same distribution as $X_j$ whenever $i<j$, we obtain that
\begin{displaymath}
\begin{split}
\mathrm{Var}[\tilde{F}_{X,y,N}(s)]&=\frac{4}{N^2}\mathrm{Var}\left(\sum_{1\le i<j\le N}\bm{1}_{\left[-\frac{s}{y_N},\frac{s}{y_N}\right]}(X_{i}-X_{j})\right) \\
&\ll \frac{1}{N^2}\sum_{1\le i<j\le N}\mathrm{Var}\left(\bm{1}_{\left[-\frac{s}{y_N},\frac{s}{y_N}\right]}(X_{i}-X_{j})\right) \\
&= \frac{1}{N^2}\sum_{1\le j\le N}(j-1)\mathrm{Var}\left(\bm{1}_{\left[-\frac{s}{y_N},\frac{s}{y_N}\right]}(X_{j})\right)\\
&= \frac{1}{N^2}\left(O(1)+\sum_{1\le k\le m}\sum_{1\le j\le I_k}(j-1)\gamma_k(1-\gamma_k)\right.\\
& \text{ } \hspace{50mm}\left. +\sum_{2^m+1\le j\le N}(j-1)\gamma_{m+1}(1-\gamma_{m+1})\right),
\end{split}
\end{displaymath}
where the last $O(1)$ comes the fact the formula \eqref{eq:lambdak} holds for all but finitely many $k$.  
Hence, with the notation of Claim \ref{claim:expectedvalue} and recalling \eqref{eq:s1_estimate} and \eqref{eq:s2_estimate}, we have that  
\begin{displaymath}
\begin{split}
\mathrm{Var}[\tilde{F}_{X,y,N}(s)] &\ll \frac{1}{N^2}\left(O(1)+ S_1+S_2+S_3 \right)\\
&\ll \frac{1}{N^2}\left(o(N)+\frac{s4^m}{y_N}+\frac{s(N^2-4^m)}{y_N}+S_3\right)\ll \frac{1}{N}+\frac{1}{N^2}\cdot S_3,
\end{split}
\end{displaymath}
where $S_3:=\sum_{1\le k\le m+1}\sum_{j\in I_k}(j-1)\gamma_k^2$. 

To conclude, recalling \eqref{eq:lambdak}, we get
\begin{displaymath}
\begin{split}
S_3&=\sum_{1\le k\le m+1} 2^{k-2}(3\cdot 2^{k-1}-1)\left(\frac{2\lfloor 2^{k+a(k)}s/y_N \rfloor+1}{2^{k+a(k)}}\right)^2\\
&\le \sum_{1\le k\le m+1} 2^{k-2}\cdot 2^{k+1}\cdot \left(\frac{2\cdot ( 2^{k+a(k)}s/y_N )+1}{2^{k+a(k)}}\right)^2\\
&\le \sum_{1\le k\le m+1} \frac{2^k}{2^{2a(k)}}\cdot \left( 4 (2^{k+a(k)}s/y_N)^2)+ 4\cdot 2^{k+a(k)}s/y_N+1\right)\\
&\ll \frac{1}{N^2}\sum_{1\le k\le m+1} 2^{3k}+\frac{1}{N}\sum_{1\le k\le m+1}\frac{2^{2k}}{2^{a(k)}} + \sum_{1\le k\le m+1}\frac{2^k}{2^{2a(k)}}\\
&\le \frac{1}{N^2}\sum_{1\le k\le m+1} 2^{3k}+\frac{1}{N}\sum_{1\le k\le m+1}2^{2k} + \sum_{1\le k\le m+1}2^k\\
&\ll N+N+N.
\end{split}
\end{displaymath}
Therefore $\mathrm{Var}[\tilde{F}_{X,y,N}(s)]\ll \frac{1}{N}+\frac{1}{N^2}\cdot S^3 \ll \frac{1}{N}$, which completes the proof.
\end{proof}


\begin{claim}\label{claim:almostthetrueone} 
\emph{\textup{(}PPC of random components along subsequences, with $s$ fixed.\textup{)}} 
Fix $s \in \RRb^+$. 
Then
$$
\mathrm{P}\left(\left\{\omega \in \Omega: \lim_{N\to \infty}\tilde{F}_{X(\omega),y, N^{2}}=2s\right\}\right)=1.
$$
\end{claim}
\begin{proof}
For each $N \in \NNb^+$, define the random variable 
$$
Z_N:\Omega \to \RRb: \omega \mapsto \tilde{F}_{X(\omega),y, N^{2}}-2s.
$$
It follows by Claims \ref{claim:expectedvalue} and \ref{claim:variance} that $\mathbb{E}[Z_N]=o(1)$ and $\mathrm{Var}[Z_N] \ll 1/N^2$ (indeed, note that the index $N^2$ in the definition of $Z_N$ above). 
%
Our thesis can be rewritten as $Z_N \to 0$ almost surely. Since $P$ is countably additive, equivalently
$$
\forall \varepsilon>0, \,\,\,\mathrm{P}\left(\left\{\omega \in \Omega: |Z_N(\omega)|>\varepsilon \text{ for infinitely many }N \in \NNb^+\right\}\right)=0.
$$

Let us fix $\varepsilon>0$. The above condition can be rewritten as $\mathrm{P}(\limsup_{N\to \infty}Q_N)=0$, where $Q_N:=\{\omega \in \Omega: |Z_N(\omega)|>\varepsilon\}$ for all $N\in \NNb^+$. 
Note that there exists $n_0$ such that 
$|\mathbb{E}[Z_n]|<\varepsilon/2$ for all $n\ge n_0$, which implies that
$$
\mathrm{P}(Q_N) \le \mathrm{P}\left(\left\{\omega \in \Omega: |Z_n(\omega)-\mathbb{E}[Z_n]|>\varepsilon/2\right\}\right)
$$
for all $n\ge n_0$. Therefore, by Chebyshev's inequality
$$
\mathrm{P}(Q_N) \ll 
\mathrm{Var}[Z_N]
\ll 1/N^2.
$$
It follows that $\sum_{N \in \NNb^+} \mathrm{P}(Q_N)<\infty$, hence the conclusion follows by the first Borel--Cantelli lemma.
\end{proof}


\begin{claim}\label{claim:ppc_sfixed} 
\emph{\textup{(}PPC of random components along full sequence, with $s$ fixed.\textup{)}} 
Fix $s \in \RRb^+$. 
Then
\begin{equation}\label{eq:claimsfixed}
\mathrm{P}\left(\left\{\omega \in \Omega: \lim_{N\to \infty}F_{X(\omega),N}=2s\right\}\right)=1.
\end{equation}
\end{claim}
\begin{proof}
Define the sequences $v=(v_n)$ and $w=(w_n)$ by
$$
v_n:=n+\lfloor \sqrt{n}\rfloor \,\,\,\text{ and }\,\,\,w_n:=\max\{n-\lfloor \sqrt{n}\rfloor, 1\}
$$
for all $n\in \NNb^+$. Note that, thanks to Claim \ref{claim:almostthetrueone}, we have 
$\tilde{F}_{X,v, N^{2}}(s)\to 2s$ and $\tilde{F}_{X,w, N^{2}}(s)\to 2s$  
as $N\to \infty$ almost surely, let us say, for all $\omega \in \Omega_0$ with $\mathrm{P}(\Omega_0)=1$. 

To conclude the proof, for all sufficiently large $N \in \NNb^+$ and $\omega \in \Omega_0$, we have
\begin{displaymath}
\begin{split}
\frac{M^{2}}{(M+1)^{2}}\,\tilde{F}_{X(\omega),v,M^2}(s)
&\le \frac{1}{N}\#\left\{1\le i\neq j\le N: \,|X_i(\omega)-X_j(\omega)| \le \frac{s}{v_{M^{2}}}\right\}\\
&\le F_{X(\omega),N}(s)\\
&\le \frac{1}{N}\#\left\{1\le i\neq j\le N: \,|X_i(\omega)-X_j(\omega)| \le \frac{s}{w_{(M+1)^{2}}}\right\}\\
&\le \frac{(M+1)^{2}}{M^{2}}\,\tilde{F}_{X(\omega),w,(M+1)^{2}}(s),
\end{split}
\end{displaymath}
where $M:=\lfloor \sqrt{N}\rfloor$. Taking the limit as $N\to \infty$, we deduce that $F_{X(\omega),N}(s) \to 2s$ for all $\omega \in \Omega_0$. 
\end{proof}


Finally, we obtain that $(X_n(\omega))$ has Poissonian pair correlations almost surely.
\begin{claim}\label{claim:ppc_almostsurelyrandom} 
\emph{\textup{(}PPC of random components along full sequence.\textup{)}} 
We have
\begin{equation}\label{eq:PPCas}
\mathrm{P}\left(\left\{\omega \in \Omega: \forall s \in \RRb^+,\,\,\,\,\lim_{N\to \infty}F_{X(\omega),N}=2s\right\}\right)=1.
\end{equation}
\end{claim}
\begin{proof}
Note that, for each $N \in \NNb^+$ and sequence $x$ in $[0,1)$, the function 
$
\RRb^+\to \RRb: s\mapsto F_{x,N}(s)
$ 
is non-decreasing. This implies that a sequence $x$ has Poissonian pair correlations if and only if there exists a relatively dense set $S\subseteq \RRb^+$ such that $\lim_{N\to \infty}F_{x,N}(s)=2s$ for all $s \in S$. 
Since $\mathrm{P}$ is countably additive and $\RRb^+$ is separable, then  \eqref{eq:PPCas} follows by the fact that \eqref{eq:claimsfixed} holds for all fixed values of $s \in \RRb^+$, thanks to Claim \ref{claim:ppc_sfixed}.
\end{proof}

The random components give PPC, as desired. However, using only the random components would give too many different gap sizes. This is related to the fact that the gap distribution of the Poisson process is exponential, and that, accordingly, relatively large gap sizes are possible.

Hence, we introduce the "deterministic blocks." To this aim, let $b: \NNb^+ \to \NNb^+$ be another weakly increasing function such that $\lim_{m\to \infty}b(m)=\infty$ (that will be chosen later), and define 
\begin{equation}\label{eq:bm}
B_m:=\left\{\frac{j}{2^{b(m)}}: j=0,1,\ldots,2^{b(m)}-1\right\}\,\,\,\text{ and }\,\,\,C_m:=B_m\setminus B_{m-1}
\end{equation}
for each $m \in \NNb^+$, where by convention $B_0:=\emptyset$. 
Note that $c_m:=\# C_m=2^{b(m)}-2^{b(m-1)}$ for each $m \in \NNb^+$, where $b(0):=0$, so that $c_1+\cdots+c_m=2^{b(m)}$. 

Then, for each $m \in \NNb^+$, let $\{Y_{m,1},Y_{m,2}\ldots,Y_{m,c_m}\}$ be random variables on the same probability space $(\Omega, \mathscr{F}, \mathrm{P})$ which are Dirac measures on the values of $C_m$, and $Y_{m,1}(\omega)<\cdots<Y_{m,c_m}(\omega)$ for all $\omega \in \Omega$. 

Note that all $\{Y_{k,j}: k=1,\ldots,m, j=1,\ldots,c_k\}$ are Dirac measures on the values of $B_m$ and $c_m$ can be equal to $0$ (since the function $b$ is weakly increasing). 

Consider the sequence $(Z_n)$ of random variables where each "deterministic block" $(Y_{m,j}: j=1,\ldots,c_m)$ is inserted between the "random blocks" $(X_i: i \in I_{m-1})$ and $(X_i: i \in I_m)$, so that it starts as 
$$
X_1,\,\, Y_{1,1},\ldots,Y_{1,c_1},\,\,X_2,\,\, Y_{2,1},\ldots,Y_{2,c_2},\,\, X_3,X_4, \,\,Y_{3,1},\ldots,Y_{3,c_3},\,\,X_5,\ldots,X_8,\,\,Y_{4,1},\ldots 
$$
To be explicit, the sequence $(Z_n)$ is defined by:
\begin{enumerate}[label={\rm (\roman{*})}]
\item $Z_1=X_1$;
\item $Z_2=Y_{1,1},\ldots,Z_{2^{b(1)}+1}=Y_{1,c_1}$;
\item $Z_{2^{b(1)}+2}=X_2$;
\item $Z_{2^{b(m-1)}+2^{m-1}+j}=Y_{m,j}$ for all integers $m\ge 2$ and $j=1,\ldots,c_m$;
\item $Z_{2^{b(m)}+i}=X_i$ for all integers $m\ge 2$ and $i \in I_m$.
\end{enumerate}

To ease the notation in the rest of the proof, let $(D_m: m\ge 1)$ and $(R_m: m\ge 0)$ be the set of indexes of deterministic blocks and random ones, respectively, so that 
$R_0=\{1\}$, $D_1:=\{2,3,\ldots,2^{b(1)}+1\}$, $R_1:=\{2^{b(1)}+2\}$, 
$
D_m:=\{2^{b(m-1)}+2^{m-1}+1,\ldots,2^{b(m)}+2^{m-1}\}\text{ and }R_m:=\{2^{b(m)}+2^{m-1}+1,\ldots,2^{b(m)}+2^{m}\}
$ 
for all integers $m\ge 2$ (cf. also Figure \ref{fig:dmrm}). Finally, set 
$
\mathcal{D}:=\bigcup_{t\ge 1}D_t$ and $\mathcal{R}:=\bigcup_{t\ge 0}R_t.
$

\bigskip

\begin{figure}[!htb]
\centering
\begin{tikzpicture}
[scale=1.6]

\node (xbegb) at (.7,0){};
\node (xend) at (7.5,0){};
\draw[-latex] (xbegb) -- (xend);
\draw[thin, dashed] (0.7,0)--(0,0);

\node (r0) at (.5,.2){{\small $R_{m-1}$}};

\draw[thin](1.3,-.1)--(1.3,.1);
\node (a3n3) at (1.1,-.25){{\tiny $2^{b(m-1)}+2^{m-1}$}};

\node (r0) at (1.3+.6,.2){{\small $D_{m}$}};

\draw[thin](2.5,-.1)--(2.5,.1);
\node (a3n3) at (2.7,-.25){{\tiny $2^{b(m)}+2^{m-1}$}};

\node (r0) at (2.5+.6+1.25,.2){{\small $R_{m}$}};

\draw[thin](6.2,-.1)--(6.2,.1);
\node (a3n3) at (6.2,-.25){{\tiny $2^{b(m)}+2^{m}$}};

\node (r0) at (6.8,.2){{\small $D_{m+1}$}};

\end{tikzpicture}
\caption{A deterministic block $D_m$ and a random block $R_m$.}
\label{fig:dmrm}
\end{figure}
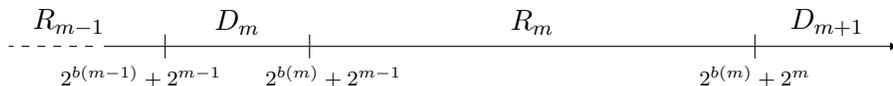

\bigskip

With these premises, we show that if the function $\tilde{b}: \NNb^+ \to \RRb$ defined by 
\begin{equation}\label{eq:btilde}
\tilde{b}(m):=m-b(m)
\end{equation}
for all $n\in \NNb^+$ is nonnegative and weakly increasing to $\infty$, then $(Z_n(\omega))$ has Poissonian pair correlation almost surely.
\begin{claim}\label{claim:ppc_almostsurelyrandomZ}  
\emph{\textup{(}PPC of random + deterministic components.\textup{)}} 
Suppose that 
the function $\tilde{b}$ defined in \eqref{eq:btilde} is weakly increasing to $\infty$ and $0\le \tilde{b}(m) \le m$ for all $m \in \NNb^+$. Then
$$
\mathrm{P}\left(\left\{\omega \in \Omega: \forall s \in \RRb^+,\,\,\,\,\lim_{N\to \infty}F_{Z(\omega),N}=2s\right\}\right)=1.
$$
\end{claim}
\begin{proof}
Thanks to Claim \ref{claim:ppc_almostsurelyrandom}, there exists $\Omega^\star\subseteq \Omega$ such that $\mathrm{P}(\Omega^\star)=1$ and
\begin{equation}\label{eq:hypothesisppc}
\forall \omega \in \Omega^\star, \forall s \in \RRb^+,\,\,\,\,\lim_{N\to \infty}F_{X(\omega),N}=2s.
\end{equation}
Hence, it is sufficient to show that, for each $\omega \in \Omega^\star$ and $s \in \RRb^+$, it holds that $\lim_{N\to \infty}F_{Z(\omega),N}=2s$ as well. Fix $\omega \in \Omega^\star$ and $s \in \RRb^+$. Note that \eqref{eq:hypothesisppc} implies that, if $(v_n)$ and $(w_n)$ are positive real sequences such that $v_n=n+o(n)$ and $w_n=n+o(n)$ as $n\to \infty$ then (we omit details)
\begin{equation}\label{eq:ppchypoequivalent}
\lim_{N\to \infty}\frac{1}{v_N}\#\left\{1\le i\neq j\le N: |X_i(\omega)-X_j(\omega)|\le \frac{s}{w_N}\right\}=2s.
\end{equation}
Here and later, suppose that $N \in D_m \cup R_m$, for some integer $m\ge 2$. 

First, let us show that, if $|Z_i(\omega)-Z_j(\omega)|\le s/N$ for some $1\le i<j\le N$, then the random variables $Z_i$ and $Z_j$ cannot be both deterministic, provided that $m$ is sufficiently large. Indeed, in such case, we would have that the minimal possible distance between (necessarily distinct) deterministic points with indexes in $[1,N] \cap \mathcal{D}$ satisfies
$$
\frac{1}{2^{b(m)}} \le \min_{\substack{1\le i< j\le N, \\ i,j \in \mathcal{D}}}|Z_i(\omega)-Z_j(\omega)| \le \frac{s}{N} \le \frac{s}{2^{m-1}+2^{b(m-1)}},
$$
which is impossible if $m$ is sufficiently large, since $\tilde{b}(m)\to \infty$ as $m\to \infty$.

Second, since 
\begin{equation}\label{eq:bothrandom}
\frac{1}{N}\#\left\{1\le i\neq j\le N: |Z_i(\omega)-Z_j(\omega)| \le \frac{s}{N} \text{ and }i,j \in \mathcal{R}\right\}
\end{equation}
can be rewritten as 
$$
\frac{1}{N}\#\left\{1\le i\neq j\le N-\#\left([1,N]\cap \mathcal{D}\right): |X_i(\omega)-X_j(\omega)| \le \frac{s}{N}\right\},
$$
and $\#\left([1,N]\cap \mathcal{D}\right) \le 2^{b(m)}=o(N)$, it follows by \eqref{eq:ppchypoequivalent} that \eqref{eq:bothrandom} has limit $2s$.

Lastly, to conclude the proof, we need to show
$$
\#\left\{1\le i< j\le N: |Z_i(\omega)-Z_j(\omega)| \le \frac{s}{N}, i \in \mathcal{R}, \text{ and }j \in \mathcal{D}\right\}=o(N).
$$
Let us suppose for the sake of contradiction that this is false. Then there exist $\delta>0$ and an infinite sequence $(N_k)$ of positive integers such that
\begin{equation}\label{eq:chaychyhptohwhcf}
\#\left\{1\le i< j\le N_k: |Z_i(\omega)-Z_j(\omega)| \le \frac{s}{N_k},i \in \mathcal{R}, \text{ and }j \in \mathcal{D}\right\}\ge \delta N_k
\end{equation}
for all $k \in \NNb^+$. Set $d_k:=\#\left([1,N_k] \cap \mathcal{D}\right)$ and let $m_k \in \NNb^+$ be the integer such that $N_k \in D_{m_k} \cup R_{m_k}$ for each $k \in \NNb^+$. In particular, 
$
2^{b(m_k-1)}\le d_k \le 2^{b(m_k)}
$ 
for all $k \in \NNb^+$. At this point, let $\eta_{k,1},\ldots,\eta_{k,d_k}$ be those elements in the index set $\mathcal{D}$ which are $\le N_k$ (note that they depend on $\omega$), and define 
$$
\nu_{k,j}:=\#\left\{1\le i\le N_k: |Z_i(\omega)-Z_{\eta_{k,j}}|\le \frac{s}{N_k} \text{ and }i \in \mathcal{R}\right\}
$$
for all $k \in \NNb^+$ and $j=1,\ldots,d_k$. Since the above sets are pairwise disjoint if $m$ is sufficiently large, it follows by \eqref{eq:chaychyhptohwhcf} that 
$
\sum_{j=1}^{d_k} \nu_{k,j} \ge \delta N_k
$ 
for all $k \in \NNb^+$. Hence, by Cauchy--Schwartz's inequality, we obtain
\begin{equation}\label{eq:contradiction}
\begin{split}
\sum_{j=1}^{d_k} \nu_{k,j}^2 \ge \frac{1}{d_k}\left(\sum_{j=1}^{d_k} \nu_{k,j} \right)^2 &\ge \frac{1}{2^{b(m_k)}}(\delta N_k)^2 \\
&\gg 
\frac{N_k^2}{2^{m_k-\tilde{b}(m_k)}} \gg 
\frac{N_k^2}{N_k / 2^{\tilde{b}(m_k)}}=N_k\cdot 2^{\tilde{b}(m_k)}.
\end{split}
\end{equation}
However, if $|Z_{i_1}(\omega)-Z_{\eta_{k,j}}|\le s/N_k$ and $|Z_{i_2}(\omega)-Z_{\eta_{k,j}}|\le s/N_k$ for some $1\le i_1\neq i_2\le N_k$ with $i_1,i_2 \in \mathcal{R}$, then 
$|Z_{i_1}(\omega)-Z_{i_2}(\omega)|\le 2s/N_k$. 
Together with \eqref{eq:contradiction}, this implies that 
$$
\frac{1}{N_k}\#\left\{1\le i\neq j\le N_k: |Z_i(\omega)-Z_j(\omega)|\le \frac{2s}{N_k} \text{ and }i,j \in \mathcal{R}\right\} \gg 2^{\tilde{b}(m_k)}\to \infty,
$$
which is contradiction since, by the argument above, the left hand side has limit $4s$ as $k\to \infty$.
\end{proof}

\begin{claim}\label{boundgaps}
\emph{\textup{(}Bounding the number of gaps.\textup{)}} 
Fix a function $q: \NNb^+\to \NNb^+$ such that $\lim_{n\to \infty} q(n)=\infty$. 
Then there exists a sequence $(x_n)$ with Poissonian pair correlations such that $g(n)\ll q(n)$ as $n\to \infty$.
\end{claim}
\begin{proof}
Note that can be assumed without loss of generality that $2\le q(n) \le 2^n$ for all $n \in \NNb^+$ and that $q$ is weakly increasing. Then define the function $h: \NNb^+ \to \NNb^+$ by $h(n)=\lfloor \log_2 q(n)\rfloor$ for all $n \in \NNb^+$ (in particular, $h(n)\le n$ and $\lim_{n\to \infty} h(n)=\infty$). At this point, define the functions $a,b: \mathbf{N}^+ \to \NNb^+$ by
$$
a(m)=\lceil h(m)/2\rceil \,\,\,\,\,\text{ and }\,\,\,\,\,\,b(m)=m+1-\lfloor h(m+1)/2\, \rfloor
$$
for all $m \in \NNb^+$. Then, thanks to Claim \ref{claim:ppc_almostsurelyrandomZ}, we have that almost all sequences $(Z_n(\omega))$ have Poissonian pair correlations. Pick one such  $\omega$. With the notation of Claim \ref{claim:ppc_almostsurelyrandomZ}, fix $n,m \in \NNb^+$ such that $n \in D_m \cup R_m$. Then, recalling the definitions \eqref{am} and \eqref{eq:bm}, we have the inclusions
$$
B_{m-1}\subseteq \{Z_1(\omega),\ldots,Z_n(\omega)\} \subseteq A_m.
$$
Note that each interval $\left[\frac{j}{2^{b(m-1)}}, \frac{j+1}{2^{b(m-1)}}\right]$ with endpoints in $B_{m-1}$ contains exactly 
$$
\frac{1}{2^{b(m-1)}}/\frac{1}{2^{m+a(m)}}=2^{m+a(m)-b(m-1)}
$$
consecutive intervals with endpoints in $A_m$. 
Considering that $m+a(m)-b(m-1)=h(m)+O(1)$, it follows that 
\begin{displaymath}
g(n) \le 2^{m+a(m)-b(m-1)}\ll 2^{h(m)} \ll q(m) \le q(n),
\end{displaymath}
\end{proof}

To conclude the proof of Theorem \ref{thm:answerlarcher}, fix a function $f: \NNb^+ \to \NNb^+$ such that $\lim_{n\to \infty}f(n)=\infty$, and let $q: \NNb^+ \to \NNb^+$ be another function such that $\lim_{n\to \infty} q(n)=\infty$ and $q(n)=o(f(n))$ as $n\to \infty$. It follows by Claim \ref{boundgaps} that there exist a constant $c>0$ and a sequence $(x_n)$ with Poissonian pair correlations such that
$$
g(n) \le cq(n) \le f(n)
$$
for all sufficiently large $n$. This completes the proof. 
\end{proof}


\subsection{Acknowledgements.} The authors are thankful to Salvatore Tringali (Hebei Normal University, CHN) 
for several comments regarding the exposition of the article.

\bibliographystyle{amsplain}
\bibliography{ppc}

\providecommand{\MR}[1]{}
\providecommand{\bysame}{\leavevmode\hbox to3em{\hrulefill}\thinspace}
\providecommand{\MR}{\relax\ifhmode\unskip\space\fi MR }
\providecommand{\MRhref}[2]{%
  \href{http://www.ams.org/mathscinet-getitem?mr=#1}{#2}
}
\providecommand{\href}[2]{#2}
\begin{thebibliography}{10}

\bibitem{MR3822226}
I.~Aichinger, C.~Aistleitner, and G.~Larcher, \emph{On quasi-energy-spectra,
  pair correlations of sequences and additive combinatorics}, Contemporary
  computational mathematics---a celebration of the 80th birthday of {I}an
  {S}loan. {V}ol. 1, 2, Springer, Cham, 2018, pp.~1--16. \MR{3822226}

\bibitem{MR3703937}
C.~Aistleitner, T.~Lachmann, and F.~Pausinger, \emph{Pair correlations and
  equidistribution}, J. Number Theory \textbf{182} (2018), 206--220.
  \MR{3703937}

\bibitem{MR3736514}
C.~Aistleitner, G.~Larcher, and M.~Lewko, \emph{Additive energy and the
  {H}ausdorff dimension of the exceptional set in metric pair correlation
  problems}, Israel J. Math. \textbf{222} (2017), no.~1, 463--485, With an
  appendix by Jean Bourgain. \MR{3736514}

\bibitem{MR3340360}
D.~El-Baz, J.~Marklof, and I.~Vinogradov, \emph{The distribution of directions
  in an affine lattice: two-point correlations and mixed moments}, Int. Math.
  Res. Not. IMRN (2015), no.~5, 1371--1400. \MR{3340360}

\bibitem{MR3336607}
\bysame, \emph{The two-point correlation function of the fractional parts of
  {$\sqrt{n}$} is {P}oisson}, Proc. Amer. Math. Soc. \textbf{143} (2015),
  no.~7, 2815--2828. \MR{3336607}

\bibitem{MR2462280}
D.~H. Fremlin, \emph{Measure theory. {V}ol. 2}, Torres Fremlin, Colchester,
  2003, Broad foundations, Corrected second printing of the 2001 original.
  \MR{2462280}

\bibitem{MR3673633}
S.~Grepstad and G.~Larcher, \emph{On pair correlation and discrepancy}, Arch.
  Math. (Basel) \textbf{109} (2017), no.~2, 143--149. \MR{3673633}

\bibitem{HinrichsEco}
A.~Hinrichs, L.~Kaltenb\"{o}ck, G.~Larcher, W.~Stockinger, and M.~Ullrich,
  \emph{On a multi-dimensional poissonian pair correlation concept and uniform
  distribution}, Monatsh. Math., to appear
  (\href{https://doi.org/10.1007/s00605--019--01267--9}{doi.org/10.1007/s00605--019--01267--9}).

\bibitem{LarcherSurvey}
G.~Larcher and W.~Stockinger, \emph{On pair correlation of sequences}, preprint
  (\href{https://arxiv.org/abs/1903.09978}{arXiv:1903.09978}).

\bibitem{LarcherMathProc}
\bysame, \emph{Pair correlation of sequences $(\{a_n\alpha\})_{n \in
  \mathbb{N}}$ with maximal order of additive energy}, Math. Proc. Cambridge
  Philos. Soc., to appear
  (\href{https://doi.org/10.1017/S030500411800066X}{doi.org/10.1017/S030500411800066X}).

\bibitem{LarcherNegative}
\bysame, \emph{Some negative results related to {P}oissonian pair correlation
  problems}, Discrete Math., to appear
  (\href{https://arxiv.org/abs/1803.05236}{arXiv:1803.05236}).

\bibitem{Marklof2019}
J.~Marklof, \emph{Pair correlation and equidistribution on manifolds}, Monatsh.
  Math., to appear
  (\href{https://doi.org/10.1007/s00605--019--01308--3}{doi.org/10.1007/s00605--019--01308--3}).

\bibitem{MR3706822}
J.~Marklof and A.~Str\"{o}mbergsson, \emph{The three gap theorem and the space
  of lattices}, Amer. Math. Monthly \textbf{124} (2017), no.~8, 741--745.
  \MR{3706822}

\bibitem{SteinerbergerRd}
S.~Steinerberger, \emph{Poissonian pair correlation in higher dimension}, J.
  Number Theory, to appear
  (\href{https://arxiv.org/abs/1812.10458}{arXiv:1812.10458}).

\bibitem{MR3704235}
\bysame, \emph{Localized quantitative criteria for equidistribution}, Acta
  Arith. \textbf{180} (2017), no.~2, 183--199. \MR{3704235}

\end{thebibliography}
\end{document}